\documentclass[12pt]{amsart}
\usepackage{amsfonts,latexsym,amssymb,amsthm,
}

\newtheorem{theorem}{Theorem}

\newtheorem{proposition}[theorem]{Proposition}

\theoremstyle{definition}

\theoremstyle{remark}

\DeclareMathOperator{\COV}{COV}
\DeclareMathOperator{\cov}{cov}

\newcommand{\m}{\mathfrak}

\title
[weakly initially compact implies mpcap]
{Every weakly initially ${\m m}$-compact topological space is  ${\m m}$pcap}

\author{Paolo Lipparini} 
\address{Dipartimento di Matematica\\
Vialis Ricerc\ae Scientific\ae\\
II Universit\`a di Roma (Tor Vergata)\\
I-00133 ROME 
ITALY}

\urladdr{http://www.mat.uniroma2.it/{\~{}}lipparin}



\begin{document}

\begin{abstract} 
The statement in the title solves a problem raised by T. Retta.
We also present a variation of the result in terms of $[ \mu ,\kappa ]$-compactness.
\end{abstract}

\maketitle  

Let ${\m m}$ be an infinite cardinal. 
A topological space is \emph{weakly initially ${\m m}$-compact} 
if and only if 
 every 
open  cover  of  cardinality  $\leq {\m m}$  has  a  finite  subset   with  a  dense  union.

A topological space $X$ 
   is  said   to 
be 
\emph{${\m m}$pcap} \cite{Re} 
if  every  family  of  
$\leq {\m m}$
  open  sets  in  X   has  a  complete  accumulation  point,  i. e.,   a 
point  each  neighborhood  of  which  meets  $\kappa$   members  of  the  family,  where  $\kappa$  is  the 
cardinality  of  the  family.
The acronym
 ${\m m}$pcap stands for
\emph{${\m m}$-pseudocompact   in  the  sense   of  complete   accumulation    points}.

The next Theorem solves the last problem in  \cite{Re}.

\begin{theorem} \label{solret}
For every infinite cardinal ${\m m}$,
every weakly initially ${\m m}$-compact topological space is  ${\m m}$pcap.
\end{theorem}

Before proving the theorem, we recall some known facts about the notions involved
in its statement.

The notion of weak initial ${\m m}$-compactness has been introduced by Frol{\'{\i}}k \cite{Fr} under the name
\emph{almost ${\m m}$-compactness}, and has been studied by various authors 
under various names, such as 
\emph{weak-$ {\m m} $-$ \aleph_0$-compactness}, or
 \emph{$\mathcal O$-$ [ \omega , {\m m} ]$-compactness}.
See \cite{Lst} for references.
By taking complements, it is trivial to see that
a topological space $X$ is weakly initially ${\m m}$-compact 
if and only if  the following holds.
For every sequence $( C _ \alpha ) _{ \alpha \in {\m m}  } $
 of closed sets of $X$, if,
for every finite $F \subseteq {\m m} $,
there exists a nonempty open set $O_F$ of $X$   such that   
$ \bigcap _{ \alpha \in F}  C_ \alpha \supseteq O_F$,
then  
$ \bigcap _{ \alpha \in {\m m}}  C_ \alpha \not= \emptyset $.

A topological space is said to be 
\emph{pseudo-$(\kappa,  \lambda )$-compact} \cite{CNcc}
if and only if for every 
$\lambda$-indexed sequence 
$(O_ \alpha ) _{ \alpha \in \lambda } $
of nonempty open sets of $X$, there is 
$x \in X$ such that, for every neighborhood $U$ of $x$, 
$|\{ \alpha \in \lambda \mid U \cap O_ \alpha \not= \emptyset \}| \geq \kappa $.   

T. Retta \cite[Theorem 3(d)]{Re} proved that
 a  space   is  ${\m m}$pcap   if and only if   it  is  pseudo-($\kappa$, $\kappa$)-compact     for  each  $\kappa \leq {\m m}$.

 \begin{proof}[Proof of the theorem] 
If $ \kappa \leq{\m m}$, then trivially every
weakly initially ${\m m}$-compact topological space
is  weakly initially $ \kappa $-compact.
Thus if we prove that, for every infinite cardinal 
$\kappa$, every weakly initially $ \kappa $-compact topological space
is 
pseudo-$(\kappa,  \kappa  )$-compact, then
we have that  
every
weakly initially ${\m m}$-compact topological space
is 
pseudo-$(\kappa,  \kappa  )$-compact, 
for every $ \kappa \leq{\m m}$,
and we are done by the mentioned result from 
\cite[Theorem 3(d)]{Re}.

Hence let $X$ be a weakly initially $ \kappa $-compact topological space, 
and let $(O_ \alpha ) _{ \alpha \in \kappa  } $ be a sequence of
 nonempty open sets of $X$. Let 
$S_ \omega ( \kappa )$ be the set of all finite subsets of $\kappa$.
Since  $|S_ \omega ( \kappa )|= \kappa $, we can reindex the sequence $(O_ \alpha ) _{ \alpha \in \kappa  } $
as $(O_ F ) _{ F \in S_ \omega ( \kappa )} $.
For every $\alpha \in \kappa $, let
$C_ \alpha = \overline{  \bigcup  \{ O_F \mid  F \in S_ \omega  ( \kappa  ),  \alpha \in F \}}$. For every finite subset $F$ of $\kappa$,
we have that  
$ \bigcap _{ \alpha \in F } C_ \alpha $
contains the nonempty open set $O_F$.
By   weak initial $ \kappa $-compactness,
$ \bigcap _{ \alpha \in \kappa  } C_ \alpha   \not= \emptyset$.

Let $x \in \bigcap _{ \alpha \in \kappa  } C_ \alpha  $.
We are going to show that,
for every neighborhood $U$ of $x$, 
we have that
$|\{ F \in S_ \omega ( \kappa ) \mid U \cap O_ F \not= \emptyset \}| = \kappa $,
thus $X$ is pseudo-$(\kappa,  \kappa  )$-compact,
and the theorem is proved.   

So, let $U$ be a neighborhood  of $x$,
and suppose by contradiction that the cardinality of
$H=\{ F \in S_ \omega ( \kappa ) \mid U \cap O_ F \not= \emptyset \} $
 is $  < \kappa $.
Then 
$|\bigcup H| < \kappa $. 
Choose $\alpha \in \kappa $ such that $ \alpha \not\in \bigcup H$.
Thus if $F \in S_ \omega ( \kappa )$ and $\alpha \in F$, then
$F \not\in H$, hence $U \cap O_ F = \emptyset$.
Then we also get  $U \cap \bigcup  \{ O_F \mid  F \in S_ \omega  ( \kappa  ),  \alpha \in F \} = \emptyset$,
hence $U \cap  C_ \alpha = \emptyset$,
since
$C_ \alpha = \overline{  \bigcup  \{ O_F \mid  F \in S_ \omega  ( \kappa  ),  \alpha \in F \}}$.
In particular, $x \not\in C_ \alpha $. 
 We have reached a contradiction, and the theorem is proved.
\end{proof}  

In fact, our argument gives something more.
Let us say that a topological space is \emph{weakly $[ \lambda ,\kappa ]$-compact} 
if and only if 
 every 
open  cover  of  cardinality  $ \leq \kappa $  has  a  subset 
of cardinality $< \lambda $   with  a  dense  union.
This notion has been studied in \cite{LiF,Lst}, sometimes under the name
$\mathcal O$-$ [ \lambda, \kappa  ]$-compactness. 

For $ \kappa  \geq \lambda  \geq \mu $,
let $\COV( \kappa , \lambda ,\mu )$ denote the minimal cardinality of a family of
subsets of $ \kappa $, each of cardinality  $< \lambda  $, such that every subset of
$ \kappa $ of cardinality $<\mu $ is contained in at least one set of the family.
Highly non trivial results about 
$\COV( \kappa , \lambda ,\mu )$ are proved in \cite{Sh}
under the terminology  $ \cov(\kappa, \lambda ,\mu,2)$.
See \cite[II, Definition 5.1]{Sh}.
  Notice that, trivially,  
$\COV( \kappa , \lambda ,\mu )\leq |S_ \mu ( \kappa )|=\sup _{\mu'<\mu} \kappa ^{\mu'}   $.
In particular, $\COV( \kappa , \lambda , \omega )= \kappa $, hence the next Proposition 
 is stronger than Theorem \ref{solret},
via \cite[Theorem 3(d)]{Re}.

\begin{proposition} \label{gen}
Suppose that
$ \kappa \geq \lambda  \geq \mu $
are infinite cardinals, and either $\kappa> \lambda $, or $\kappa$  is regular.
Then every  weakly $[ \mu ,\kappa ]$-compact topological space is
pseudo-$(\kappa, \COV( \kappa , \lambda ,\mu ))$-compact.
 \end{proposition} 

 \begin{proof} 
The proof is essentially the same as the proof of Theorem 
\ref{solret}.
We shall only point out the differences.
Let $K$ be a subset of  $S_ \lambda ( \kappa )$
witnessing $|K|=  \COV( \kappa , \lambda ,\mu )$. 
Suppose that $X$ is a weakly $[ \mu ,\kappa ]$-compact topological space
and let $(O_ Z) _{ Z \in K  } $ be a sequence of
 nonempty open sets of $X$.
For $ \alpha \in \kappa $, put 
$C_ \alpha = \overline{  \bigcup  \{ O_Z\mid  Z \in K,  \alpha \in Z \}}$.
If $W \subseteq \kappa $, and $|W|< \mu $, then
there is $Z \in K$ such that $Z \supseteq W$, 
so that
$ \bigcap _{ \alpha \in W } C_ \alpha \supseteq 
\bigcap _{ \alpha \in Z } C_ \alpha $
contains the nonempty open set $O _{Z}$, hence, 
by  weak $[ \mu ,\kappa ]$-compactness,
$ \bigcap _{ \alpha \in \kappa  } C_ \alpha   \not= \emptyset$.

Now notice that
the union of $<\kappa$ sets, each of cardinality $<\lambda$,
has cardinality $< \kappa $, and this is the only thing that is used in the 
final part 
of the proof of Theorem \ref{solret}. 
\end{proof} 

For $\kappa$ a regular cardinal, 
 weak $[ \kappa ,\kappa ]$-compactness is equivalent to
pseudo-$(\kappa,  \kappa )$-compactness,
as proved  in \cite{LiF} under different terminology. 

By replacing everywhere nonempty open sets by points in 
Proposition \ref{gen}, we get the following
result which, in the present generality, might be new.

\begin{proposition} \label{genp}
Suppose that
$ \kappa \geq \lambda  \geq \mu $
are infinite cardinals, and either $\kappa> \lambda $, or $\kappa$  is regular,
 and let $\nu=\COV( \kappa , \lambda ,\mu )$.
If $X$ is a $[ \mu ,\kappa ]$-compact topological space, then,
for every  $\nu$-indexed family $(x _ \beta ) _{ \beta \in \nu} $
of elements of $X$, there is some element $x \in X$
such that, for every neighborhood $U$ of $x$, the set  
$\{ \beta \in \nu \mid x_ \beta \in U\}$ has cardinality $ \geq \kappa $. 
 \end{proposition} 

A common generalization of both Propositions \ref{gen} and \ref{genp}
can be given along the abstract framework presented in \cite{LiF,Lst}.  
If $X$ is a 
topological space, and $\mathcal F$ is a family of subsets of $X$, we say that
$X$ is $\mathcal F$-$[ \mu ,\kappa ]$-compact if and only if 
the following holds.
For every sequence $( C _ \alpha ) _{ \alpha \in \kappa   } $
 of closed sets of $X$, if,
for every $Z \subseteq \kappa $ with $|Z|<\mu$,
there exists a set $F_Z \in \mathcal F$    such that   
$ \bigcap _{ \alpha \in Z}  C_ \alpha \supseteq F_Z$,
then  
$ \bigcap _{ \alpha \in \kappa }  C_ \alpha \not= \emptyset $.

\begin{proposition} \label{genf}
Suppose that
$ \kappa \geq \lambda  \geq \mu $
are infinite cardinals, and either $\kappa> \lambda $, or $\kappa$  is regular,
 and let $\nu=\COV( \kappa , \lambda ,\mu )$.
Suppose that $X$ is a 
topological space, and $\mathcal F$ is a family of subsets of $X$.
If $X$ is $\mathcal F$-$[ \mu ,\kappa ]$-compact, then,
for every  $\nu$-indexed family $(F _ \beta ) _{ \beta \in \nu} $
of elements of $\mathcal F$, there is some element $x \in X$
such that, for every neighborhood $U$ of $x$, the set  
$\{ \beta \in \nu \mid F_ \beta \cap U \not= \emptyset  \}$ has cardinality $ \geq \kappa $. 
 \end{proposition} 

Proposition \ref{gen} is the particular case of Proposition \ref{genf} when we take
$\mathcal F$ to be the family of all nonempty subsets of $X$.  
Proposition \ref{genp} is the particular case of Proposition \ref{genf} when we take
$\mathcal F$ to be the family of all singletons of $X$.

\end{document}